\documentclass[letterpaper, 10 pt, conference]{ieeeconf}

\IEEEoverridecommandlockouts

\usepackage{amsmath, amsfonts, mathtools, ulem}
\usepackage{hyperref}
\usepackage[usenames, dvipsnames]{color}

\usepackage{graphics} 
\usepackage{times} 
\usepackage{amssymb}  
\usepackage{cancel} 

\title{\LARGE \bf
Bang--Bang Optimal Control of Vaccination in \\ Metapopulation Epidemics with Linear Cost Structures \\
{\small \it Published in IEEE Control Systems Letters (L-CSS) at \url{https://ieeexplore.ieee.org/document/11018532} }
}

\author{Lucas~M.~Moschen and M.~Soledad~Aronna
\thanks{Manuscript received on March 17, 2025.
L. Moschen is funded by EPSRC DTP in Mathematical Sciences Grant No. EP/W523872/1 and supported by the ICL-CNRS Lab. 
The current research was mainly developed while the first author was a Master's student at FGV EMAp, with institutional financial support. 
M. Aronna is funded by the Brazilian agencies FAPERJ (processes E-26/203.223/2017 and E-26/201.346/2021) and CNPq (process 312407/2023-8).}
\thanks{For the purpose of open access, the author has applied a Creative Commons Attribution (CC-BY) license to any Author Accepted Manuscript version arising.}%
\thanks{Published in IEEE Control Systems Letters (L-CSS): \url{https://ieeexplore.ieee.org/document/11015582}}%
\thanks{© 2025 IEEE. Personal use of this material is permitted. 
Permission from IEEE must be obtained for all other uses, in any current or future media, including reprinting/republishing this material for advertising or promotional purposes, creating new collective works, for resale or redistribution to servers or lists, or reuse of any copyrighted component of this work in other works.}
\thanks{L.~M.~Moschen is with the Department of Mathematics, Imperial College London, SW7 2AZ, UK (e-mail: lucas.moschen22@imperial.ac.uk).}%
\thanks{M.~S.~Aronna is with the School of Applied Mathematics, Funcação Getulio Vargas (FGV EMAp), Rio de Janeiro, Brazil (e-mail: soledad.aronna@fgv.br).}%
}
\date{\today}

\newcommand{\R}{\mathbb{R}}

\newcommand{\fori}[1]{#1 \in \{1,\dots, K\}}

\newtheorem{assumption}{Assumption}
\newtheorem{lemma}{Lemma}
\newtheorem{proposition}{Proposition}
\newtheorem{theorem}{Theorem}
\newtheorem{remark}{Remark}

\begin{document}

\maketitle
\thispagestyle{empty}
\pagestyle{empty}

\begin{abstract}
This paper investigates optimal vaccination strategies in a metapopulation epidemic model.
We consider a linear cost to better capture operational considerations, such as the total number of vaccines or hospitalizations, in contrast to the standard quadratic cost assumption on the control.
The model incorporates state and mixed control-state constraints, and we derive necessary optimality conditions based on Pontryagin's Maximum Principle. 
We use Pontryagin's result to rule out the possibility of the occurrence of singular arcs and to provide a full characterization of the optimal control.
\end{abstract}

\section{Introduction}

Recent outbreaks of Ebola, Influenza, and COVID-19 have underscored the vulnerability of humanity to infectious diseases.
Mathematical models provide a framework for quantifying epidemic damage, forecasting, preventing, and mitigating future outbreaks.
In particular, vaccination has been shown to reduce the number of new cases and alleviate public health burdens.
However, optimal vaccination policies require tractable models that capture real-world distribution and administration constraints.

This paper studies optimal vaccination strategies in a general epidemic model defined over a metapopulation of $K$ distinct groups, capturing heterogeneous transmission dynamics and resource constraints. 
Metapopulation epidemic models, rooted in ecological patch models~\cite{levins1969some}, have been adapted for epidemiological applications~\cite{rvachev1985mathematical}, with a focus on direct graph structures and infections occurring at the nodes~\cite{sattenspiel1995structured, arino2003multi, colizza2007reaction}.
Recent work has enriched these models with time-varying parameters, demographic effects, and control strategies \cite{liu2013transmission, stolerman2015sir, yin_novel_2020}.

Motivated by these models, several authors have formulated optimal control problems using mainly quadratic costs for the control function \cite{ogren2002vaccination, asano2008optimal, rowthorn2009optimal}.
The single-patch linear-cost model has been completely characterized by~\cite{behncke2000optimal}.
Recent studies have benchmarked numerical schemes for metapopulation models \cite{salii2022benchmarking} and developed computational frameworks for COVID-19 control \cite{lemaitre2022optimal, nonato_robot_2022}.
Alongside these computational advances, the mathematical treatment of constrained optimal control in epidemiology has been considered with integral, state, and mixed control-state constraints~\cite{hansen2011optimal}.
Theoretical results on constrained SEIR models are presented in \cite{biswas2014seir}, while a control-affine SEIR model with time-specific vaccine limits is introduced in \cite{ de2015optimal}.
Recent contributions~\cite{avram2022optimal, trabelsi2023constrained} further refined these approaches by incorporating hospital capacity and supply constraints.

While most optimal vaccination models use quadratic cost functions for their convexity properties and ease of analysis \cite{gaff2009optimal,NeilanLenhart2010}, real-world vaccination costs often scale linearly or sublinearly due to economies of scale \cite{gandhi2014updating, portnoy2020producing}. 
This is because, as the number of vaccines administered increases, fixed costs (e.g., cold-chain infrastructure, training, logistics setup) are spread over more units.
We therefore investigate how this linear cost structure influences optimal vaccination strategies on metapopulations, demonstrating a bang--bang control behavior.
The main contributions of this paper are as follows.
\begin{itemize} 
    \item Proof of the existence of optimal solutions for the constrained vaccination control problem in a heterogeneous, multi-group epidemic model \eqref{eq:optimal_control_problem}; 
    \item Analysis of the optimal control problem with linear vaccination cost, including a formal proof that the optimal policy is bang--bang across $K$ regions;
    \item Complete characterization of the optimal control structure, with numerical experiments illustrating the theoretical results. 
\end{itemize}

The remainder of the paper is organized as follows. 
\autoref{sec:modeling} formulates the metapopulation epidemic vaccination model as an optimal control problem, and \autoref{sec:optimality} presents necessary optimality conditions for it.
\autoref{sec:main_results} is dedicated to the main results, namely the absence of singular arcs, the bang-bang structure, and other technical lemmas. 
\autoref{sec:numerics} presents numerical experiments confirming our theoretical results, and \autoref{sec:conclusion} concludes with possible extensions.

\noindent {\it Notation.} We write $\R_{\ge0}=[0,\infty)$ and $\|\cdot\|_{\infty}$ for the essential-supremum norm.  
The space $L^\infty([0,T];\R^k_{\ge0})$ consists of (equivalence classes of) essentially bounded measurable maps $u:[0,T]\to\R^k_{\ge0}$.  
If $g:\R^n\to\R^m$, then $D_x g$ or $\bigl[\partial g_j/\partial x_i\bigr]_{j,i}$ denotes its $m\times n$ Jacobian; for any matrix $M$, $M_{ij}$ is its entry in row $i$, column $j$.  
Given a set $S$, $\bar S$ is its closure, $\inf S$ its infimum, and $\sup S$ its supremum.  
Finally, ``a.e.'' abbreviates ``almost everywhere.''

\section{General Epidemic Model Formulation}
\label{sec:modeling}

We consider a heterogeneous population partitioned into $K$ compartments (e.g., by spatial position or age), each one of size $n_i > 0$, for $\fori{i}$.
The epidemic evolves over a finite time horizon $[0, T]$ and, for each group $\fori{i}$, we define at time $t$:
\begin{itemize}
    \item the proportion of susceptible individuals: $s^{(i)}(t) \in \mathbb{R}$;
    \item the proportions of individuals in compartments of other stages of the disease (e.g. exposed, infectious): $x^{(i)}(t) = (x_1^{(i)}(t),\dots, x_d^{(i)}(t))^\top \in \mathbb{R}^d;$ and
    \item the proportion of vaccinated in group $i:$ $V^{(i)}(t).$
\end{itemize}
Set $\boldsymbol{s} = (s^{(1)}, \dots, s^{(K)})$, $\boldsymbol{x} = (x^{(1)} | \dots | x^{(K)}) \in \mathbb{R}^{d\times K}$ and the total vaccinated population $V(t) \coloneqq \sum_{i=1}^{K} n_i V^{(i)}(t)$.

The epidemic dynamics, subject to vaccination rates $u_i \colon [0,T] \to \mathbb{R}_{\ge 0}$, is governed by
\begin{equation}
    \label{eq:system_epidemic}
    \begin{aligned}
        \dot{s}^{(i)} &= -s^{(i)} f_i(\boldsymbol{x}) -  u_i s^{(i)}, \\
        \dot{x}^{(i)} &= s^{(i)} f_i(\boldsymbol{x}) e_1 + g(x^{(i)}), \\
        \dot{V} &= \sum_{i=1}^K u_i n_i s^{(i)},
    \end{aligned}
\end{equation}
where $e_1 =(1,0,\dots,0)^\top \in \mathbb{R}^d$ (meaning that people that exit the susceptible compartment $s^{(i)}$ enter in the group corresponding to the first component of $x^{(i)}$), and the functions $ f_i : \mathbb{R}^{d \times K} \to \mathbb{R}$ and $g : \mathbb{R}^{d} \to \mathbb{R}^d$ are twice continuously differentiable.
The function $g$ models within-group disease progression (for instance, passage from the exposed to infectious compartments), while $f_i$ defines the force of infection per susceptible in group $i$.
The nonlinear term $s^{(i)} f_i(\boldsymbol{x})$ captures the infection process due to temporary interregional interactions through which susceptible people from region $i$ become infected. 
Examples of models with this structure include classical epidemiological models such as SIR and SEIR with vaccination (both single-group and stratified versions), as well as more complex variants, discussed in~\cite{lemaitre2022optimal,nonato_robot_2022,aronna2024optimal}.

We assume that the effect of new births and deaths during the time interval of study is negligible, hence we suppose the total population is constant, a condition enforced by
\begin{equation}
    \label{eq:constant_population}
    \sum_{j=1}^d g_j(x) = 0, \text{ for all } x \in \R^d,
\end{equation}
where $g_j$ refers to the $j$th component of $g$.
Additionally, $g(0) = 0$ guarantees that the disease-free equilibrium is invariant.
For positivity, we assume $f_i(\boldsymbol{x}) \geq 0$ if $\boldsymbol{x} \ge 0$ and $g_i$ is {\it essentially non-negative}, this is, if $x\in \R^d$, $x \ge 0$, and $x_j = 0$, then $g_j(x) \geq 0$.

\begin{remark}\label{rem:qj}
    Propositions \ref{prop:positive_invariance_control}, \ref{prop:existence_model} and \ref{prop:nondegeneracy} remain valid if the susceptible dynamics include any Lipschitz-continuous migration term \(q_i(\mathbf{s})\) to the $s^{(i)}$-dynamics.
    The bang--bang proof, however, relies on monotonicity, so extending Theorems \ref{th:singular_arcs} and \ref{th:optimal_vaccination} to include $q_i$ is left for future work.
    Numerical tests with linear $q_i$ (see \autoref{fig:q_j_nonzero}) also show one switch per week, suggesting the qualitative behavior is robust.
\end{remark}

\begin{proposition}\label{prop:positive_invariance_control}
    Under the standing assumptions, the system \eqref{eq:system_epidemic}, with initial data $s^{(i)}(0)>0$, $x^{(i)}_j(0)\ge 0$, $V^{(i)}(0)\ge0,$ for all $\fori{i}$ and $j \in \{1,\dots,d\}$, admits a unique bounded, absolutely continuous solution for each $u\in L^{\infty}([0,T];\mathbb{R}^K_{\ge0})$ that remains in $\mathbb{R}^{K(d+1)+1}_{\ge 0}$.
    Moreover, $s^{(i)}(t) \ge l_i > 0,$ for all $t \in [0,T]$, $i\in \{1,\dots,K\}$.
\end{proposition}

\begin{proof}\label{proof:positive_invariance_control}
    For a given $u \in L^\infty([0,T]; \R^m_{\ge 0})$, define the state vector
    \(
    w(t) \coloneqq (\boldsymbol{s}(t), \boldsymbol{x}(t), V(t))
    \)
    so that system~\eqref{eq:system_epidemic} can be written in compact form as
    \[
    w'(t) = F(t, w(t)), \quad w(0) = w_0.
    \]
    The existence and uniqueness of a solution for the latter system follow from Carath\'eodory's theorem~\cite{coddington1955theory} since one has:
    (i) for each fixed $w$, $F(\cdot, w)$ is measurable, 
    (ii) for each fixed $t$, $F(t, \cdot)$ is continuously differentiable (hence locally Lipschitz), and 
    (iii) $F(\cdot, w)$ is locally integrable due to the essential boundedness of $u$. 
    To ensure non-negativity, \autoref{lemma:positivity} implies that the non-negative orthant $\mathbb{R}^{K(d+1)+1}_{\geq 0}$ is forward invariant. 
    Finally, since $q_i$ is essentially non-negative, $f_i$ and $u_i$ are bounded on $[0,T]$, there exists a constant $C_i>0$ such that $\dot{s}^{(i)} \ge -C_i s^{(i)}$.
    By Gronwall's inequality, it follows that $s^{(i)}(t)$ is bounded from below by $l_i \coloneqq s^{(i)}(0)e^{-C_i T} > 0$.
    Additionally, $w$ is bounded since the population is constant over time.
\end{proof}

Healthcare limitations for each group $\fori{i}$, leads to the {\it mixed control-state constraint}
\begin{equation}
    \label{eq:mixed_constraints}
    0 \le u_i(t) s^{(i)}(t) \le v_i^{\max},
\end{equation}
where $v_i^{\max}$ is the maximum vaccination rate in group $i$. 
Since $s^{(i)}(t) \geq l_i > 0$ for each $i$, by \autoref{prop:positive_invariance_control},
\begin{equation}
    \label{eq:compactness_u}
    u_i(t) \leq v_i^{\max}/l_i,
\end{equation}
which ensures compactness and convexity of the admissible control set.
Vaccines arrive in weekly shipments from a centralized authority, with maximum weekly supply $V_w^{\max}$ for each $w=0,\dots, \frac{T}{7}-1$ (assuming $T$ is a multiple of 7), giving the cumulative shipment function
\[
D(t) \coloneqq \sum_{n=0}^{w}V_n^{\max}, \quad \text{for }\, t\in [7w,7(w+1)).
\]
We approximate $D$ by a twice continuously differentiable function $D_{\epsilon}$ with parameter $\epsilon>0$ satisfying
\[
D_{\epsilon}(t)=D(t),\quad \text{for }\, 7w+\epsilon/2\le t\le 7(w+1)-\epsilon/2.
\]
We then enforce the following {\it pure-state constraint} 
\begin{equation}
    \label{eq:state_constraints}
    V(t) \leq D_{\epsilon}(t), \quad \text{for }\,  t \in [0,T],
\end{equation}
allowing unused vaccines from a week to be allocated later.

The cost functional accounts for hospitalization and vaccination costs in the following way
\begin{equation}
    \label{eq:cost_function}
    J[u] \coloneqq c_v V(T) + \sum_{i=1}^{K} \int_{0}^{T} n_i f_0\bigl(x^{(i)}(t)\bigr)\,dt,
\end{equation}
where $c_v > 0$ and $f_0: \mathbb{R}^d \to \mathbb{R}_{\ge0}$.
Here, $f_0$ maps the disease-stage vector to a per-individual health cost.

The resulting {\bf optimal control problem} is
\begin{equation}
    \label{eq:optimal_control_problem}
    \begin{aligned}
        \min_{u} \quad & J[u], \quad \text{s.t.} \quad \eqref{eq:system_epidemic}, \eqref{eq:mixed_constraints}, \eqref{eq:state_constraints}, u_i(t) \geq 0.
    \end{aligned}
\end{equation}

We look for $u$ to minimize the finite-horizon cost $J[u]$ on $[0,T]$ with no terminal constraint, so our focus is on the open-loop solution of the optimization problem \eqref{eq:optimal_control_problem} rather than on the closed-loop optimizer or stabilizer.  
Although the uncontrolled system admits a disease-free equilibrium (and may admit endemic equilibria depending on $g$), their presence does not affect the optimization.

\begin{proposition}
    \label{prop:existence_model}
    The optimal control problem \eqref{eq:optimal_control_problem} admits a global minimum.
\end{proposition}

\begin{proof}
    We check directly that all the assumptions of Cesari's existence theorem
    \cite[Th.~1]{cesari1965existence} are fulfilled by problem
    \eqref{eq:optimal_control_problem}.
    Let $w(t) \coloneqq (\boldsymbol{s}(t), \boldsymbol{x}(t), V(t))$.

    First, define the pure-state constraint
    \(g(t,w)\coloneqq V-D_\varepsilon(t)\) and the set
    \[
    \Omega\;=\;\{(t,w)\in[0,T]\times\R_{\ge0}^{K(d+1)+1}\;:\;g(t,w)\le0\}.
    \]
    Because \(D_\varepsilon\) is continuous and the state space is bounded by the population balance, \(\Omega\) is compact.  
    By \autoref{prop:positive_invariance_control} every admissible trajectory of \eqref{eq:system_epidemic} with \(V\le D_\varepsilon\) stays in~\(\Omega\).

    For \((t,w)\in\Omega\), set
    \[
    U(t,w)\;=\;\Bigl\{u\in\R_{\ge0}^{K}\;:\;m(w,u)\le0\Bigr\},
    \]
    with $m(w,u)\;=\;\bigl(u_i\,s^{(i)}-v_i^{\max}\bigr)_{i=1}^{K}$.
    Since \(s^{(i)}(t)\ge l_i>0\), each component of \(u\) is bounded by
    \(v_i^{\max}/l_i\); hence \(U(t,w)\) is a closed, bounded and convex box, and the continuity of \(m\) implies upper semi-continuity of the map \((t,w)\mapsto U(t,w)\).

    The \(C^2\)-regularity of \(f_i\) and \(g\) makes \(F\) a Carathéodory mapping, locally Lipschitz in \(w\) and continuous in \((w,u)\).
    The running cost is \(\ell(w)=\sum_{i=1}^{K} n_i f_0(x^{(i)})\),
    and the Bolza terminal term \(c_v V(T)\) is converted into the Lagrange form \(\psi(w,u)=c_v\sum_{i=1}^{K} n_i u_i s^{(i)}\).
    Both \(\ell\) and \(\psi\) are continuous and affine in \(u\).

    For Cesari's convexity requirement introduce
    {\small \begin{align}
    \tilde Q(t,w)=\Bigl\{(z^0,z)\in\R\times\R^{K(d+1)+1}\Bigm|\;
    & z^0\ge \ell(w)+\psi(w,u),\nonumber\\
     z = F(w,u),\; &u\in U(t,w)\Bigr\}.
    \end{align}}
    Because \(u\mapsto(\ell(w)+\psi(w,u),F(w,u))\) is affine and \(U(t,w)\) is convex, each \(\tilde Q(t,w)\) is convex in \((z^0,z)\).

    All the hypotheses of \cite[Th.~1]{cesari1965existence} are therefore verified: the state set is compact and forward invariant, \(F\) is Carathéodory, the admissible control map is compact, convex and upper semi-continuous, the integrand is continuous and convex in \(u\), and the sets \(\tilde Q(t,w)\) are convex.
    Consequently, an optimal admissible pair \((w^\ast,u^\ast)\) exists for problem~\eqref{eq:optimal_control_problem}.
\end{proof}

\section{Necessary Conditions for Optimality}
\label{sec:optimality}

We derive necessary optimality conditions using Pontryagin's Maximum Principle for problems with state, mixed, and end-point constraints.
Under our standing assumptions, \cite[Theorem~3.1]{arutyunov2011maximum} guarantees that an optimal process $(\boldsymbol{s}^*, \boldsymbol{x}^*, V^*, u^*)$ satisfies Pontryagin's Principle, provided that the endpoint, state, and mixed constraints satisfy appropriate regularity conditions; moreover, compatibility between the state and endpoint constraints is ensured by \(V(0)=0\le D_{\epsilon}(0)\) and by explicitly imposing \(V(T)\le D_{\epsilon}(T)\). 
For brevity, we omit the star in the optimal process throughout the paper.

Define the Hamiltonian function:
\begin{equation}
    \begin{aligned}
    H &\coloneqq -\sum_{i=1}^K \psi^s_i \left(s^{(i)}\left(f_i(\boldsymbol{x}) + u_i\right) \right) \\
    &\quad +\sum_{i=1}^K \langle \psi^{x^{(i)}}, s^{(i)} f_i(\boldsymbol{x}) e_1 + g(x^{(i)}) \rangle + \psi^V \sum_{i=1}^K u_i n_i s^{(i)} \\
    &\quad - \lambda \left(\sum_{i=1}^K u_i n_i s^{(i)} - D_{\epsilon}'(t)\right) - \lambda_0 \sum_{i=1}^K n_i f_0(x^{(i)}).
    \end{aligned}
\end{equation}
The maximum principle states that there exist absolutely continuous adjoint variables $\psi_i^s:[0, T]\to\mathbb{R}$, $\psi^{x^{(i)}}: [0, T]\to\mathbb{R}^d$, for $\fori{i}$, $\psi^V: [0 ,T]\to\mathbb{R},$ together with a function $\lambda:[ 0,T] \to \mathbb{R}$, a scalar $\lambda_0 \geq 0$, and a measurable, essentially bounded function $r: [0, T] \to \R^{2K}$, such that they verify the following {\it non-triviality condition}
\begin{equation}
    \lambda_0 + \|\psi^s\|_{\infty} + \sum_{i=1}^{K} \|\psi^{x^{(i)}}\|_{\infty} + \|\psi^V\|_{\infty} + \|\lambda\|_{\infty} > 0, 
    \label{eq:nondegeneracy}
\end{equation}
solve the following system of ordinary differential equations
\begin{equation}
    \begin{aligned}
        \dot\psi_i^s &= (\psi^s_i - \psi^{x^{(i)}}_1) f_i(\boldsymbol{x})
        - ((\psi^V - \lambda) n_i - \psi_i^s) u_i + r_i u_i \\
        \dot\psi^{x^{(i)}} &= \sum_{j=1}^K (\psi_j^s - \psi_1^{x^{(j)}}) s^{(j)} \left(\frac{\partial f_j(\boldsymbol{x})}{\partial x^{(i)}} \right)^{\top}
        \\ 
            &\quad - \psi^{x^{(i)}} D_x g(x^{(i)}) + \lambda_0 n_i \nabla f_0\left(x^{(i)}\right)^{\top}, \\
        \dot\psi^V &= 0,
    \end{aligned}
    \label{eq:adjoint_system}
\end{equation}
with endpoint conditions
\begin{equation}
    \psi^s_i(T) = \psi^{x^{(i)}}_j(T) = 0, \quad \psi^V(T) = -\lambda_0 c_v,
\end{equation}
for $i\in\{1,\dots,K\}$ and $j\in\{1,\dots,d\}$.

Moreover, for each region $\fori{j}$, the optimality conditions imply that
\begin{equation}
    \label{eq:optimality_condition}
    ((\psi^V(t) -\! \lambda(t))n_j -\! \psi_j^s(t)) s^{(j)}(t) = r_j(t) s^{(j)}(t) -\! r_{j+K}(t).
\end{equation}
with $r_j(t) \ge 0$, $r_{j+K}(t) \ge 0$, and the complementarity conditions
\begin{equation}
    \label{eq:complementarity}
    r_j(t) \left(u_j(t) s^{(j)}(t) - v_j^{\max} \right) = r_{j+K}(t) u_j(t) = 0
\end{equation}
holding almost everywhere in $[0,T]$.

The function $\lambda$ satisfies the following properties:
\begin{enumerate}
    \item[(i)] $\lambda$ is constant on any interval $[s_1, s_2]$ where $V(s) < D_{\epsilon}(s)$ for all $s_1 \leq s \leq s_2$;
    \item[(ii)] $\lambda$ is left-continuous on $(0,T)$ and satisfies $\lambda(T) = 0;$
    \item[(iii)] $\lambda$ is non-increasing over $[0,T]$, ensuring that $\lambda \geq 0$.
\end{enumerate}

\section{Main Results}
\label{sec:main_results}

Based on the necessary conditions established in the previous section, we now characterize the {\it bang--bang} structure of optimal vaccination and provide some properties of the adjoint variables.
We also investigate the qualitative behavior of the optimal vaccination strategy.

Consider an optimal process $(\boldsymbol{s}, \boldsymbol{x}, V, u)$.
The adjoint equations~\eqref{eq:adjoint_system} imply that $\psi^V(t) = -\lambda_0 c_v$ for all $t \in [0,T]$.
Moreover, multiplying \eqref{eq:optimality_condition} by $u_j$ and dividing by $s^{(j)} > 0$ yields
$
((\psi^V - \lambda) n_j - \psi_j^s) u_j - r_j u_j = 0,
$
which simplifies the adjoint equation for $\psi_i^s$ to
\[
\dot\psi_i^s = (\psi^s_i - \psi^{x^{(i)}}_1) f_i(\boldsymbol{x}). 
\]
Additionally, due to \autoref{prop:nondegeneracy} below, we can fix $\lambda_0 = 1$.

\begin{proposition}[Nondegeneracy]
    \label{prop:nondegeneracy}
    In any optimal process, we must have $\lambda_0 > 0$.
\end{proposition}

\begin{proof}
    Suppose, by way of contradiction, that $\lambda_0 = 0$. 
    Then the adjoint system becomes a homogeneous linear ODE $\dot\Psi(t)=A(t)\,\Psi(t)$ with 
    $\Psi(T)=0$, where \(\Psi=(\psi^s,\psi^x,\psi^V)^\top\) and \(A(t)\) is continuous.
    By uniqueness of solution of system~\eqref{eq:adjoint_system}, it follows that $\Psi \equiv 0$, which implies that $\psi^s_i = \psi^{x^{(i)}}_j = \psi^V = 0$ for all $i$ and $j$. 
    However, the nondegeneracy condition~\eqref{eq:nondegeneracy} then implies that $\|\lambda\|_{\infty} > 0$. 
    Due to the properties of $\lambda$, there exists a maximal interval $[0,t^*]$ with $t^* \in (0,T)$ such that $\lambda(t) > 0$ for $t \in [0,t^*]$ and $\lambda(t)=0$ for $t \in (t^*,T]$. 
    Moreover, the optimality condition~\eqref{eq:optimality_condition} simplifies to
    \[
    -\lambda(t)\,n_j\,u_j(t) = r_j(t)\,u_j(t),\quad \text{for a.e. } t\in[0,T].
    \]
    If $u_j(t)>0$ for some $j$, then we must have $\lambda(t)=r_j(t)=0$. 
    Thus, for almost every $t \in [0,t^*]$, it follows that $u_j(t)=0$ for all $j$. 
    This implies that, over $[0,t^*]$, we have $V(t)=0$ (and hence $V(t) < D_\epsilon(t)$), so that $\lambda$ is constant over this interval. 
    In other words,
    \[
    \lambda(t)=
        \begin{cases}
        a, & t\le t^*,\\[1mm]
        0, & t>t^*,
        \end{cases}
    \]
    for some constant $a>0$. 
    By the continuity of $V$, there exists $\delta>0$ such that $V(t)<D_\epsilon(t)$ for all \(t\in[0,t^*+\delta]\). 
    In particular, this implies that $\lambda$ must remain constant over $[0,t^*+\delta]$, yielding $a=0$, which contradicts the standing hypotheses.
    Therefore, we conclude that $\lambda_0>0$.
\end{proof}

Intuitively, $\lvert\psi^s_i\rvert<\psi^{x^{(i)}}_1$ since, under the shadow-price interpretation~\cite{behncke2000optimal}, infection incurs a higher cost than susceptibility.
\autoref{lemma:shadow_price} proves this fact rigorously under the following simplifying assumptions.

\begin{assumption}
    \label{assumption:linearity_assumption}
    \begin{enumerate}
        \item[(a)]  For each $\fori{j}$, the functions $f_j$ are linear functions of $\boldsymbol{x}$;
        In particular, setting  $\beta_{ij} \coloneqq \nabla_{x^{(i)}} f_j(\boldsymbol{x})$, we assume that $\beta_{ij} \ge 0$ for each $i,j$, meaning that non-susceptibles from region $i$ do not negatively affect the transmission rate.
        \item[(b)] The function $f_0$ is linear, with $c \coloneqq \nabla f_0(x) \ge 0$;
        \item[(c)] The function $g$ is linear.
        Since $g$ is essentially negative, its Jacobian $G \coloneqq D_x g$ is Metzler;
        \item[(d)] The state-dependent cost is nontrivial. 
        In particular, if we define
        \[
        \begin{split}
            I &\coloneqq \Bigl\{ j \in \{2,\dots, d\} \;\Bigm|\; \exists\, m \ge1, (i_k)_{k=0,\dots,m} \text{ s. t. } \\ 
            &\quad\quad i_0=i, i_m=j,  \text{ and } G_{i_{k+1},i_k} > 0, \text{ for all } k\Bigr\}.
        \end{split}
        \]
        then we assume that \(c_j > 0\) for some \(j\in I\).
    \end{enumerate}
\end{assumption}

\begin{remark}
    \label{remark:assumption}
    Item  (d) in~\autoref{assumption:linearity_assumption} is imposed solely to exclude the trivial case where the optimal control $u_i\equiv 0$ for all $i$.
    If $c_j = 0$ for every $j \in I$, then the state variable corresponding to the first component of $x^{(i)}$ would not affect the cost, not even indirectly, yielding $u_i \equiv 0$ for all $i$. 
    In other words, vaccination must affect at least one cost-contributing state; otherwise, the optimal control is the trivial one.
\end{remark}

\begin{lemma}
    \label{lemma:shadow_price}
    Under \autoref{assumption:linearity_assumption} and the condition that $f_i(\boldsymbol{x}(t)) > 0$ for all $t \in [0,T]$, the adjoint variables of any optimal process satisfy, $\fori{i}$, and $t \in [0, T)$, 
    \[
    \psi^{x^{(i)}}_1(t) < \psi_i^s(t) <0 ,\quad \dot{\psi}_i^s(t) > 0,\quad \dot{\psi}_1^{x^{(i)}}(t) \ge 0.
    \]
\end{lemma}

\vspace{5pt}

\begin{proof}
    See the Appendix (page~\pageref{proof:shadow_price}).
\end{proof}

To determine when vaccination is active in each region, we introduce the {\it switching function}, which governs the activation of control inputs and is given by
\begin{equation}
    \label{eq:switching_function}
    \phi_j(t) \coloneqq (c_v + \lambda(t))n_j + \psi_j^s(t).
\end{equation}
From the optimality condition~\eqref{eq:optimality_condition}, the optimal control law satisfies the following conditions
\[
\begin{cases}
    \text{if }\phi_j(t) > 0, \text{ then } u_j(t) = 0; \\
    \text{if }\phi_j(t) < 0, \text{ then } u_j(t) s^{(j)}(t) = v_j^{\max}; \\
    \text{if }\phi_j(t) = 0, \text{ then } u_j(t) s^{(j)}(t) \in [0, v_j^{\max}].
\end{cases}
\]
Moreover, the function $\phi$ is left-continuous, which leads us to the following theorems that characterize the bang--bang structure of the optimal control with at most one switch per week.
Under the assumptions of \autoref{lemma:shadow_price}, we prove Theorems~\ref{th:singular_arcs} and~\ref{th:optimal_vaccination}.

\begin{theorem}
    \label{th:singular_arcs}
    For every $\epsilon > 0$ and \(w = 0, \dots, \frac{T}{7}-1\), the optimal controls contain no singular arcs. 
    More precisely,  there is no open interval \(\Lambda \subseteq W_{w,\epsilon} \coloneqq [7w+\epsilon, 7(w+1)-\epsilon]\) such that
    \[
    0 < u_i s^{(i)} < v_i^{\max}\quad \text{a.e. on }\, \Lambda.
    \]
\end{theorem}
\vspace{0.2cm}
\begin{proof}
    Fix $\epsilon > 0$ and a week $w$.
    Suppose, by way of contradiction, that there exists an open interval $\Lambda \subseteq W_{w,\epsilon}$ and an index $i$ such that $u_i(t) s^{(i)}(t) \in (0, v_i^{\max})$ for almost every $t \in \Lambda$.
    Then, by the complementarity conditions~\eqref{eq:complementarity}, we have $r_i(t) = r_{i+K}(t) = 0$.
    Thus, the optimality condition~\eqref{eq:optimality_condition} implies that $\phi_i(t) = 0$ for all $t \in \Lambda$.
    Next, we observe that $\dot{V}(t) \ge n_i u_i(t) s^{(i)}(t) > 0$ for almost every $t \in \Lambda$, which implies that $V(t) < D_{\epsilon}(t)$ on $[\inf \Lambda, \sup \Lambda)$ since $D_{\epsilon}$ is constant over $W_{w,\epsilon}$.
    Hence, the multiplier $\lambda(t)$ is constant on $[\inf \Lambda, \sup \Lambda - \delta]$ for every $\delta > 0$ since it is constant over any closed interval where the state constraint is inactive.
    By the left-continuity of $\lambda$, it is constant over $\bar{\Lambda}$. 
    By \autoref{lemma:shadow_price}, $\phi_i(t)$ must be strictly increasing on $\bar{\Lambda}$, a contradiction with $\phi_i(t) = 0$ on $\Lambda$.
    Therefore, no such open interval $\Lambda$ exists, and there are no singular arcs on $W_{w,\epsilon}$.
\end{proof}

The core of the proof of \autoref{th:singular_arcs} relies on the fact that the switching functions are strictly increasing in the presence of singular arcs, a property that follows from the strict monotonicity of the adjoint variables $\psi_i^s$.
Consequently, extending the result depends on establishing the positivity of the adjoint system~\eqref{eq:adjoint_system}.

\begin{theorem}[Optimal vaccination structure]\label{th:optimal_vaccination}
    Fix \(\epsilon>0\) and \(w\in\{0,\dots,\frac{T}{7}-1\}\).
    Then, for any week $W_{w,\epsilon}$, the optimal strategy satisfies the properties listed below.
    \begin{enumerate}
        \item If vaccination is administered at any point during the week (that is, $u_i s^{(i)} > 0$), then it must have commenced at the start of the week;
        \item If vaccination is inactive at any point in a week (that is, $u_i = 0$), then it remains inactive for the rest of the week;
        \item If the vaccine stock is exhausted before the end of the week (that is, \( V(t) = D_{\epsilon}(t) \)), then vaccination ceases immediately.
    \end{enumerate}
\end{theorem}

\begin{proof}
    Fix $\epsilon>0$ and a week $w$. 
    If no time in $t \in W_{w,\epsilon}$ satisfies $V(t)=D_\epsilon(t)$, then $\lambda(t)$ is constant over $W_{w,\epsilon}$ since the interval is closed.
    By \autoref{lemma:shadow_price}, the switching function $\phi_i(t)$ is strictly increasing for every $\fori{i}$.
    Hence, for each $i$, if $u_i(t)s^{(i)}(t) > 0$ at some time $t$, then $\phi_i(\tau) \le 0$ for every $\tau \in [7w+\epsilon, t]$, which implies that $u_i(\tau)s^{(i)}(\tau) > 0$.
    Otherwise, $u_i(t) = 0$ for the whole interval $W_{w,\epsilon}$ because $\phi_i(7w+\epsilon) > 0$.
    Additionally, there is at most one point $\tau^*_i$ such that $\phi_i(\tau_i^*) = 0$.
    
    Now, suppose there exists $t\in W_{w,\epsilon}$ such that $V(t)=D_\epsilon(t)$ and define $t^* \coloneqq \inf\{t \ge 7w+\epsilon: V(t)=D_\epsilon(t)\}$.
    For $t \in  [t^*,7(w+1)-\epsilon)$, the state constraint is active, so $V$ is constant, and therefore, $u_i(t) = 0$ for all $\fori{i}$.
    If $t^* = 7w + \epsilon$, the same argument applies. 
    Alternatively, for $t\in [7w+\epsilon,t^*)$, the state constraint is inactive since $D_{\epsilon}$ is constant.
    Therefore, $\lambda$ is constant over $[7w+\epsilon,t^*-\delta]$ for every $\delta > 0$.
    The left-continuity of $\lambda$ implies it is constant over $[7w+\epsilon,t^*]$.
    In this interval, the switching function $\phi_i(t)$ is strictly increasing, yielding $u_i(t)s^{(i)}(t)=v_i^{\max}$ for almost every $t \in [7w+\epsilon,\tau_i^*]$, for some switching time $\tau_i^* \in (7w+\epsilon, t^*)$ such that $\phi_i(\tau_i^*) = 0$ if there exists a solution.
    If not, we either have $\phi_i < 0$, which implies that $u_i s^{(i)} = v_i^{\max}$ for $t \in [7w + \epsilon, t^*)$ and we set $\tau_i^* = t^*$, or $\phi_i > 0$, which implies $u_i = 0$ for the whole week and we set $\tau_i^* = 7w + \epsilon$.

    Joining both cases, we guarantee that the optimal control has the structure
    \[
    u_i(t)s^{(i)}(t)=
    \begin{cases}
    v_i^{\max}, & t\in [7w+\epsilon,\tau_i^*), \\
    0, & t\in [\tau_i^*,7(w+1)-\epsilon],
    \end{cases}
    \]
    which precisely implies the three claims of the theorem.
\end{proof}

\autoref{th:optimal_vaccination} {\it fully characterizes the optimal control structure} by showing that within each week, the control is {\bf bang--bang with at most one switching time} $\tau_i^*,$ for each $\fori{i}$.
This bang--bang structure indicates that the optimal strategy is to {\bf vaccinate at maximum capacity until reaching the supply constraint} (i.e., until exhausting the weekly vaccine allocation).
However, this may not be feasible in practice, since it depends on individuals presenting for vaccination; hence, a more realistic implementation would be opening the vaccination facilities at maximum capacity until the vaccines for that week are over.
This gives an easily implementable practical solution that other general continuous time-varying controls would not have provided. 
A consequence of this finding is that, rather than optimizing over all admissible controls, we only need to optimize the switching times, that is, the optimal control problem reduces to optimizing \(K\cdot(T/7)\) variables -- one for each region and week -- provided that $\epsilon$ is chosen sufficiently small.

\section{Numerical experiments}
\label{sec:numerics}

We validate our theoretical findings by simulating optimal vaccination strategies in three cases (networks with $3$, $5$, and $8$ cities, each city following a simple SIR model with daily commuting
between them).
All simulations share a 28-day horizon (\(T=28\)), recovery rate
\(\gamma=1/7\) day\(^{-1}\), and home-stay fraction
\(\alpha=0.64\) (workers spend \(1-\alpha\) of each week in the
destination city).  
City-specific transmission rates \(\beta_i\in[0.10,0.35]\), populations, shipment limits \(V^{\max}_w\), and vaccination capacities \(v_i^{\max}\).
The specific values are given in the Appendix (page~\pageref{appendix:parameters}).

\autoref{fig:optimal_controls} displays the optimal control profiles and the proportion of infected individuals for these cases. 
In all examples, for each week, vaccination is administered at the maximum rate until a unique switching time, after which it is abruptly stopped, which agrees with our theoretical findings.

We solve the optimal control problem~\eqref{eq:optimal_control_problem} with a {\it first discretize, then optimize\/} approach using the Gekko Python library~\cite{beal2018gekko}, which calls IPOPT for non-linear optimization.
All experiments were performed on a Mac equipped with an Apple M3 processor (8 cores) and 16 GB of memory.

\begin{figure}[!ht]
  \centering
  \includegraphics[width=\linewidth]{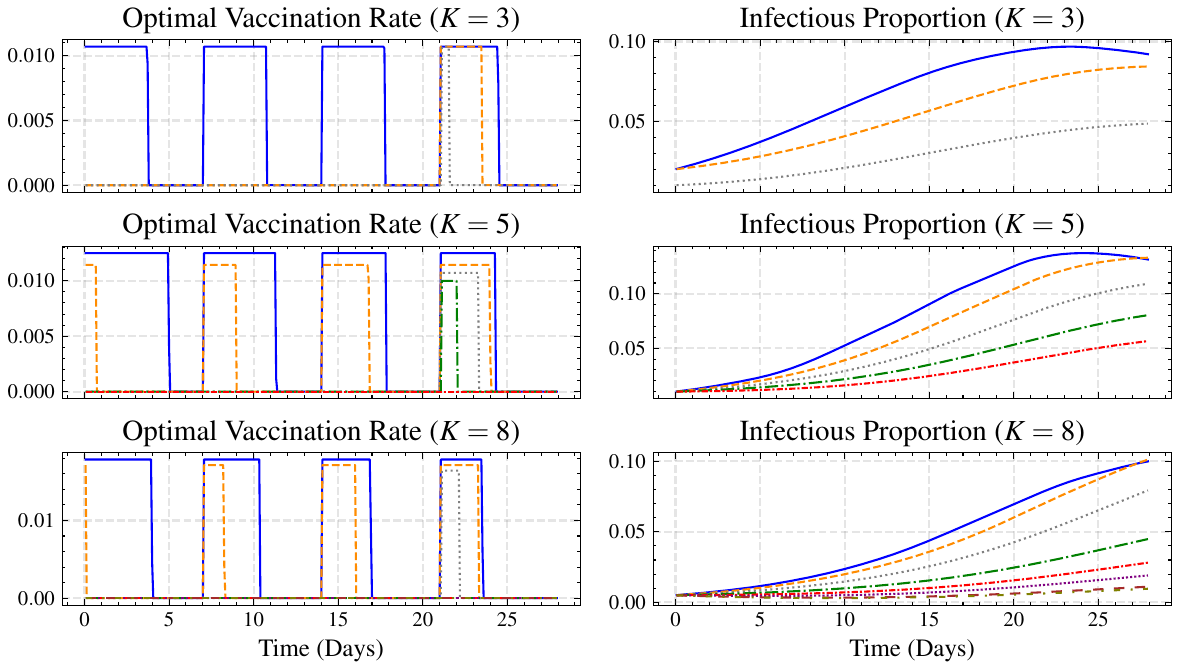}
  \caption{Optimal weekly vaccination profiles (left) and resulting infectious proportion (right) for metapopulation SIR networks with $K=3,5,8$ cities.
  The legend is omitted as the optimal control structure is of interest rather than the specific city identities.}
  \label{fig:optimal_controls}
\end{figure}
 
We further test the impact of nonzero and linear migration terms $q_j$ (violating Assumption 1) in two experiments with $K=3$ and $K=5$ cities.
\autoref{fig:q_j_nonzero} indicates that the adjoint variables are strictly increasing, which corroborates Lemma 1. 
This result suggests that even with nonzero $q_j$, the bang--bang structure remains intact, and Theorem 2's result can also be observed.
We also solved the problem, but with a quadratic vaccination cost.
\autoref{fig:figure2} indicates that the quadratic objective yields a smoother control that approximates the bang-bang profile; the resulting cost difference is negligible (bottom-right panel). The control resulting from the quadratic cost is harder to implement in practice than the bang-bang given in the linear-cost case.

\begin{figure}[!htbp]
  \centering
  \includegraphics[width=\linewidth]{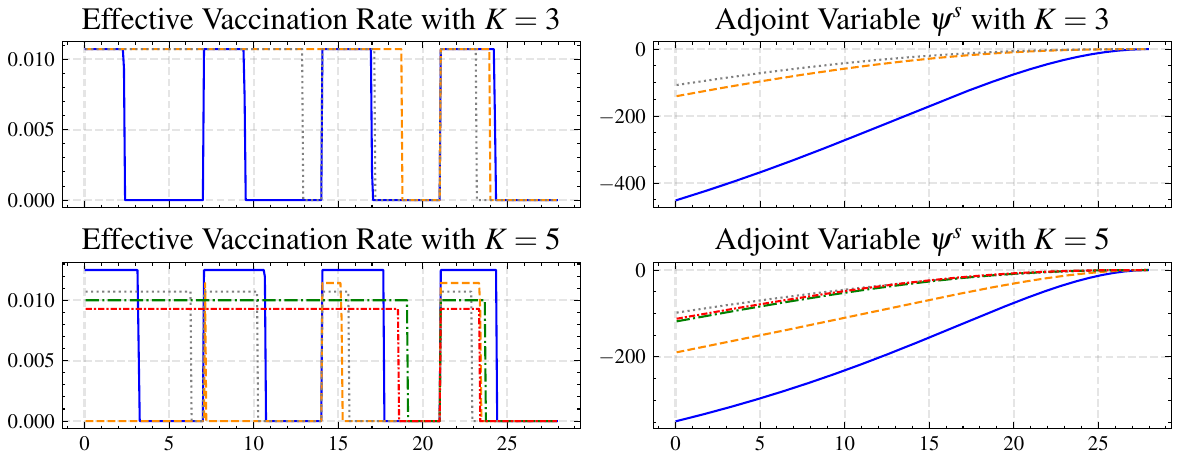}
  \caption{Optimal vaccination control profiles and adjoint variables $\psi^s$ for two network structures with $K=3$, and $5$ in the presence of linear and non-zero functions $q_j$.}
  \label{fig:q_j_nonzero}
\end{figure}

\begin{figure}[!htbp]
\centering
\includegraphics[width=\linewidth]{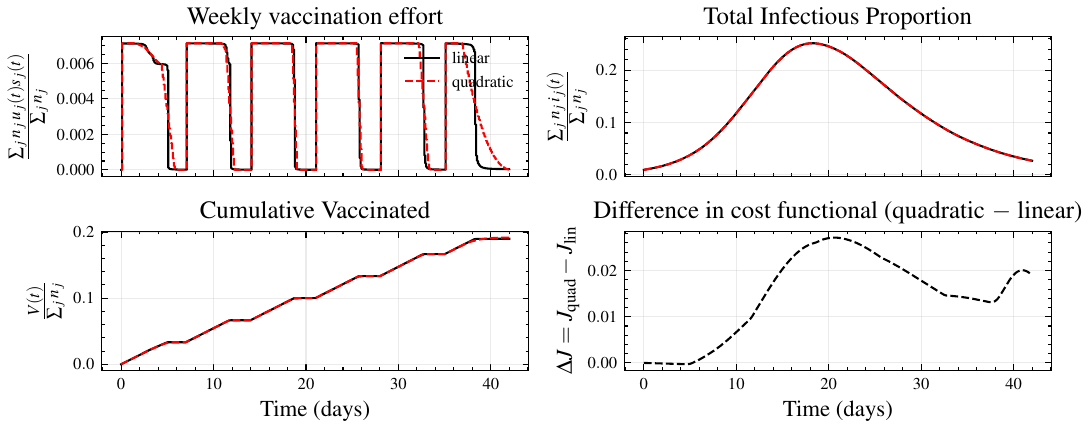}
\caption{Linear vs.\ quadratic cost for the 5-city case.
Top-left: weekly vaccination effort; top-right: total infectious proportion; bottom-left: cumulative vaccinations; bottom-right: running difference $J_{\text{quad}}-J_{\text{lin}},$ where $J_{\text{lin}}$ is the linear cost \eqref{eq:cost_function}, $J_{\text{quad}}$ replaces $V(T)$ by $\smash{\sum_j\!\int_0^T n_j u_j^2 s_j^2}$.
The coefficient \(c_v\) in the quadratic case is scaled so that the total vaccination cost is comparable to the linear one, making the cost difference illustrative.}
\label{fig:figure2}
\end{figure}

\section{Conclusion}
\label{sec:conclusion}

This work introduces a metapopulation epidemic model with vaccination that incorporates both state and mixed constraints with a linear cost.
In particular, \autoref{prop:positive_invariance_control}, \autoref{th:singular_arcs}, and \autoref{th:optimal_vaccination} are essential to establish positivity and invariance properties and fully characterize the bang--bang structure of optimal control with a finite number of switching points.
Consequently, the optimal control problem reduces to optimizing over the switching times only, which promises more efficient and reliable numerical algorithms.
Future work could focus on relaxing \autoref{assumption:linearity_assumption} to expand the applicability of the model and build an efficient algorithm to compute optimal switching times.
Analysis of closed-loop stability of infectious compartments is another interesting direction for future work.

\section*{APPENDIX}

\section*{Proofs of Auxiliary Results}
\label{sec:appendixProofs}

This appendix contains proofs of auxiliary results, including lemmas and supporting propositions not in the main text. 

\begin{lemma}
\label{lemma:positivity}
    Let $ f:[0,T] \times \mathbb{R}^n \times \mathbb{R}^m \to \mathbb{R}^n $ be continuous, with $ f(t,\cdot,u) $ locally Lipschitz for each $ t $ and $ u $.
    Fix a control \(u\in L^{\infty}([0,T];\R^{m}_{\ge0})\) and assume the solution \(x(\cdot)\) of
    \[
    x'(t)=f\bigl(t,x(t),u(t)\bigr),\qquad x(0)=x_{0}\in\R^{n}_{\ge0},
    \]
    exists on the entire interval \([0,T]\).
    If \(f\) is \emph{essentially non-negative}, i.e.
    \[
    x\ge0,\;x_{i}=0\;\Longrightarrow\;f_{i}(t,x,u)\ge0,
    \]
    then \(x(t)\in\R^{n}_{\ge0}\) for every \(t\in[0,T]\); in other words, \(\R^{n}_{\ge0}\) is forward-invariant.
\end{lemma}

\begin{proof}
    The proof of the lemma follows from standard results on dynamical systems and is omitted here.
\end{proof}

Now we present the proof to \autoref{lemma:shadow_price}.

\begin{proof}\label{proof:shadow_price}
Define the variables $z^{(i)}(t) \in \R$, $y^{(i)}(t) \in \R^d$, $w^{(i)}(t) \in \R^d$, and $\varphi^{(i)}(t) \in \R^d$ in the following way: 
$z^{(i)}(t) \coloneqq \psi_i^s(T-t)-\psi_1^{x^{(i)}}(T-t)$, $y^{(i)}(t) \coloneqq -\psi^{x^{(i)}}(T-t)^\top$, $w^{(i)}(t) \coloneqq \sum_{j=1}^K z^{(j)}(t)s^{(j)}(T-t) \beta_{ij} + c n_i,$ $\varphi^{(i)}(t) \coloneqq w^{(i)}(t) + G^T y^{(i)}(t)$,
\if{\begin{align*}
    z^{(i)}(t) &\coloneqq \psi_i^s(T-t)-\psi_1^{x^{(i)}}(T-t),  \\
    y^{(i)}(t) &\coloneqq -\psi^{x^{(i)}}(T-t)^\top, \\
    w^{(i)}(t) &\coloneqq \sum_{j=1}^K z^{(j)}(t)s^{(j)}(T-t) \beta_{ij} + c n_i, \\
    \varphi^{(i)}(t) &\coloneqq w^{(i)}(t) + G^T y^{(i)}(t),
\end{align*}}\fi
and note that the initial conditions are $z^{(i)}(0) = y^{(i)}(0) = 0$, $w^{(i)}(0) = cn_i$ and $\varphi^{(i)}(0) = cn_i$.
Using the definitions, these variables satisfy the system
\begin{align*}
    \dot{z}^{(i)}(t) &= -z^{(i)}(t) f_i(\boldsymbol{x}(T-t)) + \varphi^{(i)}_1(t), \\
    \dot{y}_i(t) &= \varphi^{(i)}(t), \\
    \dot{w}^{(i)}(t)
    &= \sum_{j=1}^K \beta_{ij} s^{(j)}(T-t) \left(\varphi^{(j)}_1(t) + u_j(T-t) z^{(j)}(t) \right), \\
    \dot\varphi^{(i)}(t) &= \dot{w}^{(i)}(t) + G^T \varphi^{(i)}(t),
\end{align*}
where we can notice that 
\[
z^{(i)}(t) = \int_0^t e^{-\int_0^\tau f_i(\boldsymbol{x}(T-\xi)) d\xi} \varphi^{(i)}_1(\tau) \, d\tau.
\]

Since $\beta_{ij}s^{(j)}(T-t) \ge 0$, $u_j(t) \ge 0$, and $G$ is a Metzler matrix, the coupled system between $z^{(i)}$ and $\varphi^{(i)}$ is essentially non-negative.
Since $z^{(i)}(0) \ge 0$ and $\varphi^{(i)}(0) \ge 0$, for all $t \in [0,T]$, by the essentially non-negative property, $z^{(i)}(t) \ge 0$ and $\varphi^{(i)}(t) = \dot{y}^{(i)}(t) \ge 0$.
Therefore, if $z^{(i)}(t^*) = 0$ for some $t^* > 0$, this implies $z^{(i)}(t) = 0$ for all $t \in [0,t^*]$.
Suppose $c_1 > 0$.
Then $\dot{z}^{(i)}(0) = c_1 n_i > 0$, which implies that there exists some $\delta > 0$ such that $z^{(i)}(t) > 0$ for all $t \in (0,\delta)$.
Consequently, $z^{(i)}(t) > 0$ for all $t > 0$.
Notice that, for all $k \in \{1,\dots, d\}$, 
\begin{equation*}
\begin{aligned}
    y_{k}^{(i)}(t) &= \int_0^t e^{G_{kk}(t-\tau)} \bigg( 
    \sum_{j\neq k} G_{jk} y_j^{(i)}(\tau) \\
    &\quad + \sum_{j=1}^K z^{(j)}(\tau) s^{(j)}(T-\tau) (\beta_{ij})_k 
    + c_k n_i \bigg) \, d\tau.
\end{aligned}
\end{equation*}
Take $l \in \{2,\dots,d\}$ with $c_l > 0$ and the sequence $1=i_0, i_1,\dots,i_m=l$ such that $G_{i_{k+1}, i_k} > 0$ for all $k \in \{0,\dots,m-1\}$ (see \autoref{remark:assumption}).
Then
\[
y_{l}^{(i)}(t) \ge \int_0^t e^{G_{ll}(t-\tau)} c_l n_i \, d\tau > 0, \quad \forall t > 0.
\]
Suppose that $y_{i_j}^{(i)}(t) > 0$ for all $t > 0$.
Then, for all $t>0$,
\[
y_{i_{j-1}}^{(i)}(t) \ge \int_0^t e^{G_{i_{j-1} i_{j-1}}(t-\tau)} G_{i_j i_{j-1} }y^{(i)}_{i_j}(\tau) \, \, d \tau > 0.
\]
By induction, we conclude that $y_1^{(i)}(t) > 0$ for all $t > 0$.

Suppose that $z^{(i)}(t^*) = 0$ for some $t^* > 0$, then $z^{(i)}(t) = 0$ for all $t \in [0,t^*]$.
Therefore, $\varphi_1^{(i)}(t) = 0$ almost everywhere in $[0,t^*]$.
By the initial condition $y_1^{(i)}(0) = 0$ with $y_1^{(i)}(t) > 0$ for $t > 0$, the Lebesgue Differentiation Theorem implies that $\varphi_1^{(i)}(\tau) > 0$ for some $\tau \in (0,t^*)$, which is a contradiction.
Therefore, $z^{(i)}(t) > 0$ for all $t > 0$.
It is straightforward to conclude that $0 > \psi_i^s(t) > \psi_1^{x^{(i)}}(t)$ for $t \in [0,T)$.
Finally, since $f_i(\boldsymbol{x}(t)) > 0$ we conclude that $\dot{\psi}_i^s(t) = z^{(i)}(T-t) f_i(\boldsymbol{x}(t)) > 0$ for $t \in [0,T)$.
\end{proof}

\section*{Simulation Parameters}
\label{appendix:parameters}

All experiments share the following global parameters:
\begin{itemize}
  \item Time horizon: $T = 28$ days
  \item Recovery rate: $\gamma = 1/7$ day$^{-1}$
  \item Home-stay fraction: $\alpha = 0.64$
  \item Weekly vaccine shipment: $V_w^{\max} = \{1, 2, 3, 4\} / 30$
  \item Cost coefficients: $c_v = 0.01$, $c_h = 100$
  \item Solver tolerances: $\texttt{RTOL} = \texttt{OTOL} = 10^{-6}$
\end{itemize}

The following parameters vary by experiment:

\textbf{3 Cities.}
\begin{itemize}
  \item $\beta = [0.3,\ 0.2,\ 0.1]$
  \item Normalized populations: [0.83, 0.083, 0.083]
  \item $s_0 = [0.96,\ 0.97,\ 0.95]$
  \item $i_0 = [0.02,\ 0.02,\ 0.01]$
  \item $v_i^{\max} = 0.3 / T$
  \item Commuting matrix:
  \[
  \begin{bmatrix}
    0.9 & 0.05 & 0.05 \\
    0.45 & 0.45 & 0.10 \\
    0.45 & 0.10 & 0.45
  \end{bmatrix}
  \]
\end{itemize}

\textbf{5 Cities.}
\begin{itemize}
  \item $\beta = [0.35,\ 0.3,\ 0.25,\ 0.2,\ 0.15]$
  \item Normalized populations: [0.5, 0.3, 0.1, 0.05, 0.05]
  \item $s_0 = [0.97,\ldots,0.97]$
  \item $i_0 = [0.01,\ldots,0.01]$
  \item $v_i^{\max} = [0.35,\ 0.32,\ 0.3,\ 0.28,\ 0.26] / T$
  \item Commuting matrix:
  \[
  \begin{bmatrix}
    0.8 &  0.05 & 0.05 & 0.05 & 0.05 \\
    0.1 &  0.7  & 0.1  & 0.05 & 0.05 \\
    0.05 & 0.1  & 0.7  & 0.1  & 0.05 \\
    0.05 & 0.05 & 0.1  & 0.7  & 0.1  \\
    0.05 & 0.05 & 0.05 & 0.1  & 0.75
  \end{bmatrix}
  \]
\end{itemize}

\textbf{8 Cities.}
\begin{itemize}
  \item $\beta = [0.35, 0.3, 0.25, 0.2, 0.15, 0.1, 0.05, 0.04]$
  \item Normalized populations: [0.46, 0.26, 0.13, 0.06, 0.03, 0.02, 0.007, 0.007]
  \item $s_0 = [0.98,\ldots,0.98]$
  \item $i_0 = [0.005,\ldots,0.005]$
  \item $v_i^{\max} \!=\! [0.5, 0.48, 0.46, 0.44, 0.42, 0.4, 0.38, 0.36] / T$
  \item Commuting matrix:
  \[
  \begin{bmatrix}
    0.7 & 0.1 & 0.05 & 0.05 & 0.03 & 0.03 & 0.02 & 0.02 \\
    0.1 & 0.7 & 0.1 & 0.02 & 0.02 & 0.02 & 0.02 & 0.02 \\
    0.05 & 0.1 & 0.75 & 0.02 & 0.02 & 0.02 & 0.02 & 0.02 \\
    0.05 & 0.02 & 0.02 & 0.75 & 0.05 & 0.05 & 0.03 & 0.03 \\
    0.03 & 0.02 & 0.02 & 0.05 & 0.8 & 0.03 & 0.02 & 0.03 \\
    0.03 & 0.02 & 0.02 & 0.05 & 0.03 & 0.8 & 0.03 & 0.02 \\
    0.02 & 0.02 & 0.02 & 0.03 & 0.02 & 0.03 & 0.85 & 0.01 \\
    0.02 & 0.02 & 0.02 & 0.03 & 0.03 & 0.02 & 0.01 & 0.85
  \end{bmatrix}
  \]
\end{itemize}

\bibliographystyle{IEEEtran}
\bibliography{biblio}

\end{document}